\documentclass{article}
\usepackage{amssymb,amsfonts,amsmath,amsthm,amscd,mathrsfs}
\usepackage{makeidx}
\usepackage[english]{babel}

\makeindex
\usepackage[x11names]{xcolor}
\usepackage{tikz-cd}  

\usepackage{hyperref}
\input xy
\xyoption{all}
\usepackage[all]{xy}

\usepackage{fancyhdr}
\usepackage{makeidx}
\makeindex 
\usepackage{vmargin} 
\setmargins{3.3cm}       
{3cm}                        
{14.5cm}                      
{23cm}                    
{10pt}                           
{0.5cm}                           
{0pt}                             
{1cm}                           

\numberwithin{equation}{section}

\newtheorem{teo}{Theorem}[section]
\newtheorem{pro}[teo]{Proposition}
\theoremstyle{definition}
\newtheorem{defi}[teo]{Definition}

\newtheorem{lem}[teo]{Lemma}
\newtheorem{ejem}[teo]{Example}

\newtheorem{rem}[teo]{Remark}
\newtheorem{coro}[teo]{Corollary}

\newtheorem{exe}[teo]{Example}
\newcommand{\m}{{}^{-1}}

\newcommand{\G}{{\mathcal G}}

\title{\textbf{Partial actions of groups  on profinite spaces}}
\author{
	Luis Mart\'inez, H\'{e}ctor Pinedo and  Andr\'{e}s Villamizar\\
	\small  Escuela de Matem\'{a}ticas\\
	\small Universidad Industrial de Santander\\
	\small Cra. 27 calle 9, Bucaramanga, Colombia\\
	\small  e-mail: luchomartinez9816@hotmail.com, hpinedot@uis.edu.co, andresvillamizar1793@hotmail.com\\
}
\date{\today}

\begin{document}

	\maketitle
	\begin{abstract} We show that  for   a partial action $\eta$ with closed domain of a compact group $G$ on a profinite space $X$ the space of orbits $X/\!\sim_G$ is profinite, this leads to the fact that when $G$ is profinite the enveloping space $X_G$ is also profinite. Moreover, we provide conditions for the induced quotient map $\pi_G: X\to X/\!\sim_G$ of $\eta$ to have a continuous section.  Relations between  continuous sections of $\pi_G$ and continuous sections of  the quotient map induced by the enveloping action of $\eta$ are also considered.        At the end of this work   we  prove that the category of actions on 
profinite spaces with countable number of clopen sets is  reflective  in the category of  actions of compact
Hausdorff spaces having countable number of clopen sets.

	\end{abstract}

	\section{Introduction}  A topological space $X$ is called 
 profinite if   there exists
an inverse system of finite discrete spaces for which its inverse limit is
homeomorphic to $X,$ equivalently,  
$X$ is profinite if it is compact Hausdorff and zero-dimensional. A topological group is profinite if it is profinite as a topological space.  Important examples of profinite groups and spaces are the group of $p$-adic integers with $p$ prime number, the Galois group on an arbitrary Galois extension, fundamental groups of connected schemes  and  the set of connected components of a compact Hausdorff space. For details on profinite spaces and profinite groups, the interested reader may consult   \cite{Magid} or \cite{RZ}. 

On the other hand, partial actions of groups  appeared in the context of $C^*$-algebras
 in  \cite{Ruy1} and \cite{Mk}, in which $C^*$-algebraic crossed products by partial automorphisms were introduced and studied by analyzing their internal structure.    Several classes of $C^*$-algebras are described as partial crossed products that correspond to partial actions of discrete groups on profinite spaces; for instance the Carlsen-Matsumoto $C^*$-algebra $\mathcal{O}_X$ of an arbitrary subshift $X$  (see \cite{DOE}) and  $C^*$-algebras of separated graphs \cite{AraE}, which is a class containing graph algebras and Cunts-Krieger algebras.

Partial actions of groups have also appeared in other contexts, such as the  theory of operator algebras, Galois cohomology, Hopf algebras, Polish spaces,  the theory of $\mathbb{R}$-trees and model theory  (see \cite{DO} and  \cite{Ruy} for a detailed account of recent developments on partial actions).  A deep application  of partial actions   is given in    \cite{AraE}  where  it was provided a counter-example  for a  conjecture of M. R{\o}rdam and  A. Sierakovski related  the  Banach-Tarski paradox.

A relevant question is whether a partial action is obtained as restriction of a corresponding collection of total maps on some superspace. In the topological context, this  is known as the globalization problem  and was studied in \cite{AB} and \cite{KL}. They showed that for any partial action $\eta=\{\eta_g\}_{g\in G}$ of a topological group $G$ on a topological space $X$ there is a topological space $Y$ and a continuous action $\mu$ of $G$ on $Y$ such that $X$ is embedded in $Y$ and $\eta$ is the restriction of  $\mu$ to $X.$  Such a space  is called a globalization of $X.$ It is also shown that there is a minimal globalization $X_G$ called the enveloping space of $X.$ However,  structural properties of $X$ are not in general inherited by $X_G,$ for instance, in \cite{Ruy2} it is shown that the enveloping space of a partial action of a countable group on a compact metric space is Hausdorff if and only if the domain of each $\eta_g$ is clopen for all $g \in  G,$ while in \cite{PU1} there were established conditions for which $X_G$ is a Polish space provided that $X$ and $G$ are. Likewise, in \cite{S}, it was proven that
the globalization of a partial action on a connected 2-complex may result in a complex which is not connected. Thus, a natural problem is to know which
properties of $X$ are preserved by $X_G.$

The present work is structured as follows: After the introduction, we provide in Section  \ref{pactions} the necessary background  and notations on (set theoretic) partial and topological partial actions, and their corresponding globalization. Also, we give some preliminary results that will be needed in the work. In Section \ref{parprof} we work with partial actions on a profinite space $X$ and present in Theorem \ref{pro3.3} a sufficient condition for the space $X/\!\sim_G$ to be profinite, where $\sim_G$ is the orbit equivalence relation determined by the partial action, as given in equation \eqref{porbit}. Later, we treat the problem of finding  a continuous section to the quotient map $\pi_G: X\to X/\!\sim_G$ and  show the existence of such a section when the group $G$ is profinite (see Theorem \ref{contgeral}) extending the classical result on continuous and free (global) actions.   It is important to notice that continuous sections have been relevant in the context of partial actions, for instance in \cite{Ruy0} continuous sections of Banach bundles play a crucial role in the characterization of  continuous twisted partial actions  of locally compact groups on $C^*$-algebras. 

The second part of Section \ref{parprof} is devoted to find relations between  continuous  sections of the quotient map $\pi_G$ and  $\Pi_G,$ respectively ( being $\Pi_G$ the corresponding quotient map induced by the globalization)  are presented in Subsection \ref{apglob}. We show how to find a continuous section of $\Pi_G$ having a continuous section of $\pi_G$.  It is important to remark that the converse does not seem to be true, that is, having a continuous section of $\Pi_G$ does not seem to imply that $\pi_G$ has a continuous section; items  (ii) and (iii)  of Proposition \ref{cor3.3}  deal with this problem (see also Proposition \ref{borel}).  We finish this work with Section  \ref{actions} which has a categorial flavor. Indeed, since the globalization problem is closely related to a reflectivity property (see  Proposition \ref{abar}), being the globalization a reflector, in  \cite{KNo} the authors studied the  problem of when the corres\-ponding reflector of the   category of  partial group actions on  sets with an algebraic structure to a subcategory of global actions  is a globalization.  In our case,  we deal with some categories related to global actions of $G$ on compact Hausdorff and profinite spaces, and show in  Theorem \ref{reft}  that the category of separately continuous actions on profinite spaces is reflective in the category of separately continuous actions on compact Hausdorff  spaces.  Under the restriction that $G$ is Hausdorff and Baire  we obtain in Proposition \ref{ref2}  that the category  whose  objects are  continuous  actions on compact Hausdorff spaces having countable clopen sets,  contains the category of   continuous  actions on profinite  spaces having countable clopen sets as a reflective subcategory.

	\section{Preliminaries on partial actions}\label{pactions} In this section  we establish our conventions on partial actions, and prove some results that will be useful throughout the work. We start with  the following.

\begin{defi} Let $ A, B, X, Y$  be sets. We say that  $f :A \to B$ is a partial function if
there exists $C \subseteq  A$ such that the restriction of $f$ to $C$ is a function. A partial set action
of $A$ on $X$  is a partial function $A \times X \to X$ given by $(a, x) \mapsto  a \cdot x ,$ for all  $a \in A$ and
$x \in X$ such that $a \cdot x$ is defined, which we denote by $\exists a\cdot x.$
\end{defi}

We proceed with the next.

\begin{defi} \cite[p. 87-88]{KL} Let $G$ be a group  with identity element $1$ and $X$ be a  set. A   partial action of $G$ on $X$ is a partial set action  $\eta$ of $G$ on $X$ such that for  each $g,h\in G$ and $x\in X$ the following assertions hold:
	\begin{enumerate}
		\item [(PA1)] If $\exists g\cdot x$, then $\exists g^{-1}\cdot (g\cdot x)$ and $g^{-1}\cdot (g\cdot x)=x$,
		\item [(PA2)] If $\exists g\cdot(h\cdot x)$, then $\exists (gh)\cdot x$ and $g\cdot(h\cdot x)=(gh)\cdot x$,
		\item [(PA3)] $\exists 1\cdot x$ and $1\cdot x=x.$
	\end{enumerate}
 It is said that $\eta$ {\it acts}  on $X$ if $\exists g\cdot x,$ for all $(g,x)\in G\times X.$
\end{defi}

Given a partial action  $\eta$  of $G$ on $X,$  $g\in G, x\in X$ and $ U\subseteq X;$ we set:
\begin{itemize} 
\item   $G*U=\{(g,u)\in G\times U\mid \exists g\cdot u\}.$ In particular, $G\ast X$ is the domain of $\eta.$

\item  $X_g=\{x\in X\mid \exists \,g\m\cdot x \}.$
\end{itemize}	
Then $\eta$ induces a family of bijections $\{\eta_g\colon X_{g\m}\ni x\mapsto g\cdot x\in X_g \}_{g\in G},$  such that $\eta_1$ is the identity of $X$ and $\eta_{g\m}=\eta\m_g.$ We also denote this family  by $\eta.$
The following result characterizes partial actions in terms of a family of bijections.
\begin{pro} {\rm  \cite[Lemma 1.2]{QR} \label{fam} A partial action $\eta$ of $G$ on $X$ is a family $\eta=\{\eta_g\colon X_{g\m}\to X_g\}_{g\in G},$ where $X_g\subseteq X,$  $\eta_g\colon X_{g\m}\to X_g$ is bijective for  all $g\in G;$ such that:
\begin{itemize}
\item[(i)]$X_1=X$ and $\eta_1=\rm{id}_X,$
\item[(ii)]  $\eta_g( X_{g\m}\cap X_h)=X_g\cap X_{gh},$
\item[(iii)] $\eta_g\eta_h\colon X_{h\m}\cap  X_{ h\m g\m}\to X_g\cap X_{gh},$ and $\eta_g\eta_h=\eta_{gh}$ in $ X_{h\m}\cap  X_{ h\m g\m},$
\end{itemize}
for all $g,h\in G.$}
\end{pro}

Given $x\in X$ and $U\subseteq X$ we set  $G^x=\{g\in G\mid \exists g\cdot x\}$   
 and $G^U\cdot U=\{g\cdot u\mid  u\in U, g\in G^u\}.$ The set $U$ is called {\it $G$-invariant} if $G^U\cdot U\subseteq U.$

We have the next.
	\begin{lem}\label{lem1.1.1}
		Let $\eta$ be  a  partial action of $G$ on $X$ and $U$ a nonempty subset of $X$,  then the following statements are true:
\begin{enumerate}
\item [(i)]  $\eta(G*U)=G^U\cdot U$,
\item [(ii)] $G^U\cdot U$ and its complement are $G$-invariant.
			
		\end{enumerate}  
	\end{lem}

\begin{proof}
		Statement (i) is straightforward.  To show  (ii) denote  $A:=G^U\cdot U$ and $B:=X\setminus A$.  The inclusion $G^A\cdot A \subseteq A$  is a consequence of  (PA2). Now we check  $G^B\cdot B\subseteq B$.  Take $y\in G^B\cdot B$, then there  are $x\in B$ and  $g\in G^x$ such that $\eta_g(x)=y$. Suppose $y\in A$ and let be $m\in U, h\in G^m$ such that $\eta_h(m)=y$. Note that $\eta_h(m)\in X_h$ and $\eta_h(m)=y\in X_g$, then $m\in \eta_h^{-1}(X_h\cap X_g)=X_{h\m}\cap X_{h\m g}$ and by (iii) of Proposition \ref{fam}  follows that $\eta_{g^{-1}h}(m)=x$, this implies that $x\in A$ which is a contradiction, therefore $y\in B$.
	\end{proof}

From now on in this work, $G$ will denote a topological group and $X$ a topological space. We endow $G\times X$ with the product topology and $G*X$ with the subspace topology. Moreover $\eta: G*X\to X$  will denote a partial action. It is said that $\eta$ is {\it topological } if every $X_g$ is open and $\eta_g$ is a homeomorphism, $g\in G;$ if moreover $\eta$ is continuous, $\eta$ 
is  called a {\it continuous partial action}.

We proceed with the next.

\begin{lem}\label{lem1.1}
		Let $\eta$ be a partial action of $G$ on $X.$ The following assertions hold:
\begin{itemize}
\item [(i)] If $\eta$ is topological  and $U$ is an open subset of $X,$ then $G^U\cdot U$ is open.   
\item  [(ii)] If $G*X$ is clopen then $X_g$ is clopen for all $g\in G.$ 
\end{itemize}
\end{lem}
\begin{proof}
(i) It follows from the fact that for every $g\in G$ the set $X_g$ is open, $\eta_g$ is a homeomorphism  and  $G^U\cdot U=\bigcup\limits_{g\in G^U} \eta_g(U\cap X_{g\m}).$ For (ii) take $g\in G$, first let us prove that $X_g$ is closed. Indeed, if $X_g=X$ this  is clear. Otherwise, take $x\in X\setminus X_g$ and since $G*X$ is closed, there exist open  sets $U\subseteq G$ and $V\subseteq X$ such that $(g^{-1},x)\in U\times V\subseteq (G*X)^c$. Moreover, if $y\in V$ then $(g^{-1},y)\notin G*X$ and $y\notin X_g$. This shows that $X_g$ is closed.
To prove that $X_g$ is open, take $x\in X_g$ then $(g^{-1},x)\in G*X$ and there are open sets $U\subseteq G$ and $V\subseteq X$ such that $(g^{-1},x)\in U\times V\subseteq G*X$ from this we get $x\in V\subseteq X_g$ and $X_g$ is open.
\end{proof}

We end this section with the following.
\begin{pro}\label{hatcon}{\rm
	Let $\eta$ be a topological partial action of $G$ on $X$, then  the family of bijections $\hat{\eta}=\{\hat{\eta}_g: (G\times X)_{g^{-1}}\rightarrow (G\times X)_g\}_{g\in G},$  where $(G\times X)_g=G\times X_g$ and 
$$\hat{\eta}_g: G\times X_{g\m}\ni (h,x) \mapsto (hg^{-1},\eta_g(x)) \in G\times X_g,$$ is a topological  partial action of $G$ on $G\times X.$  Moreover $\eta$ is continuous if and only if $\hat{\eta}$ is continuous.}
\end{pro}

\begin{proof} The first assertion is  \cite[Example 2.3]{PU1}. Now suppose that $\eta$ is continuous, set $\alpha:G*(G\times X)\ni (g, (h,x))\mapsto  hg^{-1}\in G$ and  $\beta:G*(G\times X)\ni(g,(h,x))\mapsto \eta_g(x)\in X.$
  We show that $\alpha$ and $\beta$  are continuous, for this consider $U\subseteq G$ an open set for which $\alpha(g,(h,x))\in U$ for some $(g,(h,x))\in G*(G\times X)$. Take $V$ and $W$ open subsets of  $G$ such that $(g,h)\in V\times W, WV^{-1}\subseteq U$ and  define $Z_0:= (V\times G\times W)\cap G*(G\times X)$. It is easy to see  that $(g,(h,x))\in Z_0$ and $\alpha(Z_0)\subseteq WV^{-1}\subseteq U,$ which shows that $\alpha$ is continuous. On the other hand, let $U\subseteq X$ be an open set in $X$ such that $\beta(g,(h,x))\in U$ for some $(g,(h,x))$. Since $\eta$ is continuous there are open sets $T\subseteq G$, $H\subseteq X$ such that $(g,x)\in T\times H$ and $\eta((T\times H)\cap G*X)\subseteq U.$ Let  $Z_1:=(T\times G\times H)\cap G*(G\times X)$ then $(g,(h,x))\in Z_1$ and $\beta(Z_1)\subseteq \eta((T\times X)\cap G*X)\subseteq U$. Thus, we conclude that $\hat{\eta}$ is continuous.
	
Conversely, suppose that $\hat{\eta}$ is continuous. Let $U\subseteq X$ be  an  open set  such that $\eta(g,x)\in U$ for some $(g,x)\in G*X$. By (PA3), $(g,(1,x))\in G*(G\times X)$ and $\beta(g,(1,x))=\eta(g,x)\in U$.  Since   $\beta=\pi_2\circ \hat\eta,$ where $\pi_2 : G\times X\ni (g,x)\mapsto x\in X,$ then $\beta$ is continuous and there are open subsets  $M,N\subseteq G$ and $V\subseteq X$ such that $(g,(1,x))\in M\times N\times V$ and $\beta(Z)\subseteq U$, where $Z:= (M\times N\times V)\cap G*(G\times X)$. Define $Z':= (M\times V)\cap(G*X)$ then $(g,x)\in Z'$ and  $\eta(Z')\subseteq \beta(Z)\subseteq U,$ thus $\eta$ is continuous as desired.
\end{proof}

	\subsection{Induced partial actions and globalization}	

Let  $u$ be a continuous  action of $G$ on  a  topological space $Y$ and $X\subseteq Y$ be an open set.  For $g\in G,$ set
\begin{equation}\label{induced}X_g=X\cap u_g(X)\,\,\, \text{ and }\,\,\, \eta_g=u_g{\big|}_{X_{g\m}},\end{equation}  then $\eta \colon G* X\ni (g,x)\mapsto \eta_g(x)\in X $ is a topological partial action of $G$ on $X,$  and we say that $\eta$ is  {\it induced} by  $u$ or that $\eta$ is the {\it restriction} of $u$ to $X.$

In this case, we have 
\begin{equation}\label{indug}G^x=\{g\in G\mid u_g(x)\in X\cap u_g(X)\}, \,\,{\rm for\,\, all  } \,\,x\in X.\end{equation}

\noindent An important question in the study of partial actions is whether they can be induced by  global actions. In the topological sense, this turns out to be affirmative, and a proof was given in \cite[Theorem 1.1]{AB} and independently in \cite[Section 3.1]{KL}. We recall their construction. Let $\eta=\{\eta_g: X_{g\m}\to X_g\}_{g\in G}$ be a topological partial action of $G$ on $X.$ Define an
equivalence relation $R$ on $ G\times X$ as follows:
\begin{equation}
\label{eqgl}
(g,x)R(h,y) \Longleftrightarrow x\in X_{g\m h}\,\,\,\, \text{and}\, \,\,\, \eta_{h\m g}(x)=y,
\end{equation}
and denote by  $[g,x]$   the equivalence class of the pair $(g,x).$
Consider the {\it enveloping space}  or the {\it globalization}  $X_G=(G\times X)/R$ of $X$  endowed the quotient topology. 
Then by \cite[(iii) Proposition 3.9]{KL} the action 
\begin{equation}
\label{action}
\mu \colon G\times X_G\ni (g,[h,x])\to [gh,x]\in X_G,
\end{equation}
is continuous. The map $\mu$ is called the {\it enveloping action} of $\eta.$  Further the map
\begin{equation}
\label{iota}
\iota \colon X\ni x\mapsto [1,x]\in X_G,
\end{equation} is a continuous injection such that $G\cdot \iota(X)=X_G.$ Moreover, it follows by \cite[Theorem 1.1]{AB} that $\iota$ is an open map provided that   $G\ast X$ is open. 
In this case  $\iota: X\to \iota(X)$ is a homeomorphism and one may identify $X$ with $\iota(X).$

		\section{Partial actions on profinite spaces}\label{parprof}

Given a topological partial action of $G$ on $X,$ one defines the {\it orbit equivalence relation}  $\sim_G$ on $X$ as follows:
\begin{equation}\label{porbit}x\sim_G y\Longleftrightarrow\,\, \exists g\in G^x  \,\,{\rm such \,\,that}\,\, g\cdot x=y,\end{equation}
for each $x,y\in X$. The elements of $X/\!\sim_G$ are the {\it orbits} $G^x\cdot x$ with  $x\in X$  and $X/\!\sim_G$ is endowed with the quotient topology.  It follows by \cite[Lemma 3.2]{PU1} the {\it induced quotient map of $\eta$} 
\begin{equation}\label{qmap}\pi_G: X\ni x\mapsto G^x\cdot x \in X/\!\sim_G,\end{equation} is continuous and open.  

Our next goal is to show that $X/\!\sim_G$ is profinite as long as $X$ is profinite and $G$ compact, but first for the reader's convenience we establish the following.

\begin{defi} Two points $u$ and $v$ in $X$ can be separated if  there are disjoint open subsets $U$ and $V$ of $X$ such that  $u\in U ,$ $v\in V$ and $U\cup V=X.$
\end{defi}
It is known that any two different points in a compact space $X$ can be separated if and only if $X$ is Hausdorff and is zero-dimensional (see for instance \cite[Proposition 2.3]{Magid}).
\begin{teo} \label{pro3.3}

	Let $\eta: G* X\rightarrow X$ be a  continuous partial action of a compact group $G$ on a profinite space $X$ such that $G* X$ is closed, then $X/\sim_G$ is a profinite space.
\end{teo}
\begin{proof}

	Note that $X/\!\sim_G$ is compact. Now we  show that different points $G^x\cdot x$ and $G^y\cdot y$ in $X\!/\sim_G$  can be separated.  Let  $\mathcal{C}:=\{ U\subseteq X : U$ is clopen$, \,x\in U \}$, then  $\mathcal{C}\neq \emptyset$. We claim that there exists $U\in \mathcal{C}$ such that $G^U\cdot U\cap G^y\cdot y=\emptyset$. Otherwise, for each $V\in \mathcal{C}$  the set $\tilde{F}_y(V)=\{ (g,v)\in G* V : \eta_g(v)=y\}$ is nonempty. Since $\tilde{F}_y(V)=\eta^{-1}(y)\cap (G*V)$,  then it is closed in $G*V$ and thus closed in $G\ast X$. Now, if $V_1, V_2\in \mathcal{C}$ then $\tilde{F}_y(V_1\cap V_2)\subseteq \tilde{F}_y(V_1)\cap \tilde{F}_y(V_2)$. In that sense $\lbrace \tilde{F}_y(V)\rbrace_{V\in \mathcal{C}}$ is a family in $G\ast X$ with the finite intersection property, thus there exists $(g,v)\in \bigcap_{V\in \mathcal{C}}\tilde{F}_y(V)$ this gives   $v=x$ and $\eta_g(x)=y$ which leads to a contradiction. Then there is $U\in \mathcal{C}$ such that $G^U\cdot U\cap G^y\cdot y=\emptyset$. Now we check that $G^U\cdot U$ is clopen. Indeed, it is open thanks to Lemma \ref{lem1.1};  moreover, given that $G\ast U$ is compact, $X$ is Hausdorff and $\eta(G*U)=G^U\cdot U$ then we have that $G^U\cdot U$ is closed. Finally consider $A=G^U\cdot U$ and $B=X\setminus A,$ by  Lemma \ref{lem1.1.1} the sets $A$ and $B$ are $G$-invariant and clopen, then 
$\pi_G^{-1}(\pi_G(A))=A$ and $\pi_G^{-1}(\pi_G(B))=B$ thus $G^x\cdot x$ and $G^y\cdot y$ are separated by the sets $\pi_G(A)$ and $\pi_G(B),$ respectively. 
\end{proof}

\begin{coro}\label{envpro}
	Let $\eta$ be a continuous  partial action of a profinite group $G$ on a profinite space $X.$ 
	If $G*X$ is closed, then  the enveloping spaces $X_G$ is profinite.
\end{coro}
\begin{proof}
Consider the continuous partial action $\hat\eta$ given in  Proposition \ref{hatcon} and  denote by $\sim_{\hat{G}}$ the orbit equivalence relation on $G\times X$ determined by $\hat{\eta}$, then  by \cite[Theorem 3.3]{PU1}  we get $X_G=(G\times X)/\!\sim_{\hat{G}}$. Thus, by Theorem \ref{pro3.3} it is enough to check that $G*(G\times X)$ is closed in $G \times G\times X.$
	In that sense, let  $(g,(h,x))\notin G*(G\times X)$ then  $(g,x)\notin G*X$ and   there are open sets  $T\subseteq G$ and $U\subseteq X$ such that $(g,x)\subseteq T\times U \subseteq (G*X)^c$. Note that $(g,(h,x))\in T\times (G\times U)\subseteq (G*(G\times X))^c$. This shows that $G*(G\times X)$ is closed, and  we conclude that $X_G$ is profinite.
\end{proof}

\begin{exe}  \cite[p. $22$]{Ruy} {\bf Partial Bernoulli action.} Let $G$ be a discrete group and $X:= \lbrace 0,1\rbrace^G$. There is a continuous global action $\beta=\lbrace \beta_g\rbrace$, where for all $\omega\in X$, $\beta_g(\omega)=g\omega$. The topological partial Bernoulli action $\eta$ is obtained by restricting  $\beta$ to the open set $\Omega_1=\lbrace \omega\in X : \omega(1)=1\rbrace$. Thus, by \eqref{induced} $D_g:=\Omega_1\cap \beta_g(\Omega_1)= \lbrace \omega\in X : \omega(1)=1=\omega(g)\rbrace$ and $\eta_g= \beta_g{\big|}_{D_g^{-1}}, g\in G$. Let us show that $G\ast \Omega_1$ is clopen. Let $\{(n_i, x_i)\}_{i\in I}$ be  a net in $G\ast \Omega_1$, convergent to $(n,x)$. Since $G$ is discrete then $(n_i)_{i\in I}$ is eventually constant so $n_i=n$ for large $i\in I$. On the other hand, as $x_i\longrightarrow x,$ then $1=x_i(1)\longrightarrow x(1)$ and from this $x(1)=1$. Similarly, it is obtained that $x(n^{-1})=1$ and from the above we conclude $(n,x)\in G\ast \Omega_1$. On the other hand, take $(n,x)\in G*\Omega_1$ then $x\in V=( \pi_{n^{-1}}{\big|}_{\Omega_1})^{-1}(\{1\})$ and $(n,x)\in \{n\}\times V\subseteq G*\Omega_1$. This shows that $G*\Omega_1$ is clopen. Thus, if $G$ is finite, then Theorem \ref{pro3.3} implies that $X/\!\sim_G$ is a profinite space. 

\end{exe}

\subsection{On continuous sections of  the quotient map $\pi_G.$}
In this section, we are interested in providing conditions under which the quotient map $\pi_G$ defined in \eqref{qmap} has a continuous section. For this, we start with the next.

\begin{lem}\label{ret}
	Let $\mu:G\times Y\ni  (g,y)\mapsto g\cdot y\in  Y$ be a continuous action of a topological group $G$ on a profinite space $Y.$ Suppose $X$ is a clopen subset of $Y,$ such that $G\cdot X = Y.$ Then there exists a retraction $r: Y\rightarrow X$ such that: 
\begin{enumerate}
\item [(i)]  The map  \begin{equation}\label{rre}\overline{r}: Y/\!\sim_G\rightarrow X/\!\sim_G,\,\,\, G\cdot y\mapsto G^{r(y)}\cdot r(y),\end{equation} is continuous, where, for $x\in X$ the set $G^x\cdot x$ is the orbit  of $x$ given by the induced partial action of $\mu$ on $X$ (see equation \eqref{indug}).
\item [(ii)] If $\overline{i}: X/\!\sim_G\rightarrow Y/\!\sim_G$, $G^x\cdot x\rightarrow G\cdot x$,  then  $\overline{r}\circ\overline{i}={\rm id}_{X/\sim_G}$.
\end{enumerate}
\end{lem}
\begin{proof}
 By assumption $Y=\bigcup\limits_{g\in G}g\cdot X.$ Since $Y$ is compact there are $g_1=1, g_2,\cdots, g_n\in G$ such that $Y=\bigcup\limits_{i=1}^ng_i\cdot X$. For  $1\leq j\leq n$  set $X_j=g_j\cdot X\setminus \bigcup\limits_{i=1}^{j-1}g_i\cdot X$, thus the family $\{X_j\}_{j=1}^n$  is a partition of $Y$ such that $X_1=X$ and $g_{j}^{-1}\cdot X_j\subseteq X$.  Further, the continuity of the action implies that $r_j: X_j\ni x\mapsto g_j^{-1}\cdot x\in X$ is continuous and thus $r=\bigcup\limits_{i=1}^nr_j: Y\rightarrow X$  is a retraction of $Y$ in $X$.  

\noindent (i) We first check that  $\overline{r}$ is well defined. Take $x,y\in Y$ such that  $G\cdot x=G\cdot y$, then there exist $g\in G$ and $1\leq i,j\leq n$ for which  $x=g\cdot y, x\in X_i$ and $y\in X_j$.  Then $r(y)=g_j^{-1}\cdot y$ and  
$$r(x)=g_i^{-1}\cdot x=(g_i^{-1}g)\cdot y=(g_i^{-1}g g_j)\cdot  r(y);$$ but  $r(x)\in( g_i^{-1}gg_j^{-1}\cdot X)\cap X$,  that is $g_i^{-1}gg_j^{-1}\in G^{r(x)},$ where $G^{r(x)}$ is given by \eqref{indug}. Note that $(g_jg^{-1}g_i)\cdot r(x)=r(y)$, then $G^{r(x)}\cdot r(x)=G^{r(y)}\cdot r(y)$ and  $\overline{r}$  is well defined.  To show that  $\overline{r}$ is continuous, consider the quotient map $\Pi_G:Y\rightarrow Y/\!\sim_G$ induced by $\mu.$  Since $\overline{r}\circ \Pi_G=\pi_G\circ r$  we get that $\overline{r}$ is continuous. 

\noindent (ii) It is clear that $\bar i$ is well-defined, moreover the assumption follows from the fact that $r$ is a retraction.
\end{proof}

\begin{defi}
	Let $\eta: G\ast X\longrightarrow X$ be a topological  partial action of a group $G$ on a set $X$. We say that $\eta$ is free if, for each $(g,x)\in G\ast X$ such that $\eta(g,x)=x$, we have $g=1$.
\end{defi}

\begin{teo}\label{contgeral}{\rm
	Let $\eta$ be a continuous and free partial action of a profinite group $G$ on a profinite space $X$ such that $G*X$ is clopen. Then $\pi_G : X \rightarrow X/\sim_G$ has a continuous section.}
\end{teo}
\begin{proof} Let $\mu$ be the enveloping action  of $\eta$ given by \eqref{action}. Since $\eta$ is free, then so is $\mu,$ and follows by  \cite[Proposition 2.9]{Magid}  that the corresponding quotient map $\Pi_G$ has a continuous section $q:X_G/\!\sim_G\rightarrow X_G.$ Let  $$\overline{\iota}: X/\!\sim_G\to  X_G/\!\sim_G,\,\,\, G^x\cdot x\mapsto  G\iota(x),$$ where $\iota$ is the map defined in \eqref{iota}. Note that  $\overline{\iota}\circ\pi_G=\Pi_G\circ\iota$ then $\overline{\iota}$ is continuous. Moreover, $G\iota(X)=X_G$ and $\iota(X)$ is clopen in $X_G,$ indeed it is open because $\iota$ is open, and it is closed because $\iota(X)$ is compact in the  profinite (in particular Hausdorff) space $X_G;$ it follows from Lemma \ref{ret} that  there is a retraction  $r: X_G\rightarrow X$ such that the following diagram is commutative
\begin{center}
	 $\xymatrix{ X \ar[d]^-{\pi_G} \ar[r]_-{\iota}& X_G\ar[d]_-{\Pi_G} \ar@/_{5mm}/[l]_-{r}  \\
		X/\sim_G\ar[r]_-{\overline{\iota}}\ar@/^{7mm}/[u]^-{r\circ q\circ\overline{i}}& X_G/\sim_G\ar@/^{5mm}/[l]^-{\overline{r}}\ar@/_{7mm}/[u]_-{q,}
	}$   
\end{center}

\noindent in particular, $\pi_G\circ r=\overline{r}\circ\Pi_G$. Further, by the same Lemma $\overline{r}\circ \overline{\iota}={\rm id}_{X/\sim_G}$, then
$$\pi_G\circ (r\circ q \circ \overline{\iota} )=\overline{r}\circ(\Pi_G\circ q)\circ \overline{\iota}=\overline{r}\circ \overline{\iota}={\rm id}_{X/\sim_G},$$ and the map $r\circ q \circ \overline{\iota}$ is continuous, thus  $\pi_G$ has a continuous section.
\end{proof}

\begin{rem} Notice that in general the assumption that $\eta$ acts freely on $X,$ cannot be omitted even when $\eta$ acts globally, see for instance \cite[Example 5.6.8]{RZ}.
\end{rem}

\begin{ejem}
Let $G$ be a profinite group  and  $\mu$ a free continuous action of   $G$ on a profinite space $X.$ Take  $Y$  a clopen subset of  $X,$ then  $\mu$ induces a free partial action $\eta:G\ast Y\rightarrow Y.$  It is not difficult to show that $G\ast Y$ is closed in $G\times Y.$  Moreover $G\ast Y$ is open in $G\times Y$ thanks to \cite[ii) Theorem 3.13]{KL}. 
Then follows by Theorem \ref{contgeral} that the quotient map  $\pi_G$ has a continuous section.

\end{ejem}

\subsection{Relations between continuous sections of $\pi_G$ and $\Pi_G.$}\label{apglob}
Let $\mu$ be the globalization of $\eta.$ We study relations between continuous sections of the maps $\pi_G$  and $\Pi_G,$ being   $\Pi_G$ the corresponding quotient map of  the enveloping action $\mu.$

\begin{pro}\label{cor3.3}{\rm
	Let $\eta$ be a continuous  partial action of a topological group $G$ on a  space $X.$ Then the following statements hold.
\begin{enumerate} 
\item [(i)] If $\pi_G$ has a continuous section, so does $\Pi_G$. 
\item [(ii)]  If  $G*X$ is open and $q$ is a continuous section of $\Pi_G$ such that ${\rm im}\, q\subseteq \iota(X),$ then $\pi_G$ has a continuous section.
\item [(iii)] If $\Pi_G$ and $\hat{\Pi}_G$ have continuous sections, then $\pi_G$ has a continuous section, where  $\hat{\Pi}_G$ is the quotient map $G\times X\to X_G$ induced by the partial action  $\hat \eta$ of Proposition \ref{hatcon}.
\end{enumerate}
}
\end{pro}
\begin{proof}

(i) Suppose that $q:X/\!\sim_G\rightarrow X$ is a continuous section of $\pi_G$. Consider 
$$s:\left(X_G/\sim_G\right)\ni G\cdot [1,x]\mapsto \iota(q(\pi_G(x))) \in X_G, $$ where $\iota$ is given by \eqref{iota}.  We claim that $s$ is a continuous section of $\Pi_G$. First, note that $s$ is well defined. In fact, let $x,y\in X$ such that $[1,x]\sim_G [1,y]$, we have $\pi_G(x)=\pi_G(y)$. Set $z_x=q(\pi_G(x))$ and $z_y=q(\pi_G(y))$, then $\pi_G(z_x)=\pi_G(x)=\pi_G(y)=\pi_G(z_y)$. Thus \begin{center}
		$s(G\cdot [1,x])=[1,z_x]=[1,z_y]=s(G\cdot [1,y])$,
	\end{center}
and $s$ is well-defined. Since $q$ and $\iota$ are continuous, we get that $s$ is continuous. To finish the proof, take $x\in X$ and let $y_x\in X$ such that $q(\pi_G(x))=y_x$.  Since $\pi_G(y_x)=\pi_G(x)$ there is $g\in G^x$ such that $\eta_g(x)=y_x$. Thus, $\mu(g,[1,x])=[1,y_x]$ and we have 
$$(\Pi_G\circ s)(G\cdot [1,x])=\Pi_G([1,y_x])=G\cdot [1,y_x]=G\cdot [1,x].$$ This shows that $s$ is a continuous section of $\Pi_G$.

(ii)	Let  $r:\left(X/\!\sim_G\right)\ni G^x\cdot x\mapsto  \iota^{-1}(q(G[1,x]))\in X.$  It is not difficult to check that $r$ is well-defined, moreover  that $G*X$ is open implies that $\iota$ is open and thus $r$ is continuous.  Finally,  take  $x,z_x\in X$ such that  $q(G[1,x])=[1,z_x],$ then  $G[1,x]=G[1,z_x]$ which gives $\pi_G(x)=\pi_G(z_x),$ this implies $\pi_G(r(G^x\cdot x))=\pi_G(z_x)=\pi_G(x)$ and  $r$ is a continuous  section of $\pi_G$.

(iii) Let  $q:\left(X_G/\!\sim_G\right)\to X_G$ and $t:X_G\rightarrow G\times X$ be  continuous sections of $\Pi_G$ and $\hat{\Pi}_G,$ respectively. Define 
\begin{center}
	$p:X/\!\sim_G\ni G^x\cdot x\mapsto {\rm proj}_2(t(q(G\cdot [1,x]))) \in X.$
\end{center}
We check that $p$ is a continuous section of $\pi_G$. The map  $p$ is well-defined and continuous, since $X/\sim_G\ni G^x\cdot x\mapsto G[1,x]\in X_G/\!\sim_G$ is continuous. On the other hand, take $x\in X$. Let $h,k\in G$ and $y,z\in X$ such that $q(G[1,x])=[h,y]$ and $t([h,y])=(k,z)$, then $G[1,x]=G[h,y]$ and $[k,z]=[h,y]$. Thus, $\eta_{h^{-1}g}(x)=y$ and $\eta_{k^{-1}h}(y)=z$ for some $g\in G$, thus $G^x\cdot x=G^z\cdot z,$ and we have that $\pi_G(p(G^x\cdot x))=G^z\cdot z=G^x\cdot x$, as desired.
\end{proof}

 It follows by \cite[Lemma 4.2]{PU2} that  $\Pi_G$ has a Borel section, provided that $G$ and $X$ are also a Polish group and a Polish space, respectively. Thus, by item (i) of Proposition \ref{cor3.3} and the proof  (iii) of  the same proposition,  we have  the next.
\begin{pro}\label{borel} {\rm Let $\eta$ be a continuous partial action of a profinite Polish group $G$ on a profinite Polish space  $X$ such that $G*X$ is closed, then $\pi_G$ has a Borel section.}
\end{pro}

\section{Partial actions and reflective categories}\label{actions}

Let $X$ and $Y$ be sets  equipped with partial actions by a group $G.$ A function $f:X\to Y$ is called a  \textit{G-map} if for each $g\in G$ and $x\in X$ such that $\exists g\cdot x$ then $\exists g\cdot f(x)$  and $f(g\cdot x)=g\cdot f(x).$ Let  $G$ be a topological group and  $\mathscr{PA}_G$ the category of continuous partial actions of $G$ and the $G$-maps as morphisms,  also denote by $\mathscr{A}_G$  the  subcategory  of $\mathscr{PA}_G$ whose objects are continuous actions of $G.$  

Recall the definition of a reflective subcategory.
\begin{defi} Let $\mathscr{C}$ be a category, a subcategory  $\mathscr{D}$ of  $\mathscr{C}$ is reflective, if the inclusion functor $\mathscr{D}\to \mathscr{C}$ has a left adjoint. 
\end{defi}

We have  the next.
\begin{pro}\label{care}{\rm \cite[Theorem 2 (i), P 83]{MC} A subcategory $\mathscr{D}$ of a category  $\mathscr{C}$ is  reflective, if and only if for any object
$C \in  \mathscr{C}$ there exists an object $R_D(C) \in   \mathscr{D}$  and a morphism $\epsilon_D(C):  C\to R_D(C)$
 such that for all $ D \in    \mathscr{D}$  and $\varphi : C \to D$
there is a unique   $\psi: R_D(C) \to D$ in   $\mathscr{D} $ with   $\psi\circ  \epsilon_D(C)=\varphi.$}
\end{pro}

 For the convenience of the reader we represent a partial action $\eta$ of a group $G$ on a set $X$ with the symbol $\eta^X.$ The following fact follows from  \cite[Theorem 1.1]{AB} and Proposition \ref{care}.

\begin{pro}\label{abar}{\rm  The category $\mathscr{A}_G$  is reflective in $\mathscr{PA}_G$. More precisely, the enveloping functor $E:\mathscr{PA}_G\to  \mathscr{A}_G$  sending an object $\eta^X\in Ob(\mathscr{PA}_G)$ to  $\mu^{X_G}$ defined  by \eqref{action},  and a morphism $f: \eta^X\to \theta^Y$ to $E(f): X_G\ni [g,x]\mapsto [g, f(x)]\in Y_G,$ is left adjoint to the inclusion functor.}
\end{pro}

As a consequence of the results above we obtain the next.

\begin{coro} Let $G$ be a profinite group, $\{\eta_i^{X_i}\}_{i\in I}$ be a family of objects in $\mathscr{PA}_G$ 
 such  that $\varprojlim \eta_i^{X_i}=\eta^X$ in $\mathscr{PA}_G,$ then  $\varprojlim \mu_i^{({X_i})_G}$ exists in $ \mathscr{A}_G$ and $\varprojlim \mu_i^{({X_i})_G}=\mu^{X_G},$  where $\mu^{X_G}$ is the enveloping action of $\eta^X$.  In particular, if  every $X_i$ is finite and discrete,  then $X_G$ is  profinite provide that  ${\rm dom}\, \eta$  is closed  or ${\rm dom}\, \eta_i$ are closed, for any $i\in I.$  
\end{coro}
\begin{proof} The first assertion follows from Proposition \ref{abar} and the fact that  left adjoint functors preserve inverse limits. For the last assertion, suppose first that ${\rm dom}\, \eta$  is closed. Now  $X$ is profinite because  $\varprojlim X_i=X,$   and the result follows from Corollary \ref{envpro}. Finally, if ${\rm dom}\, \eta_i$ are closed for any $i\in I$  we have by  Corollary \ref{envpro} that $(X_i)_{G}$ is profinite for any $i\in I,$ thus $X_G$ is profinite because $\varprojlim {(X_i)}_G=X_G$ is an inverse limit of profinite spaces.
\end{proof}

\subsection{From actions on $X$ to actions on $ Comp(X)$}
 We denote by  $\mathscr{SAK}_G$ the category of  separately continuous actions of a topological group $G$ on compact Hausdorff spaces  and morphisms are continuous $G$-maps. Also, let 
$\mathscr{SAP}_G$  be the subcategory of $\mathscr{SAK}_G$ whose objects are  separately continuous actions of $G$ on profinite spaces. 
 In this section, inspired by  \cite{AT}, we  prove that 
 $\mathscr{SAP}_G$ is reflective in $\mathscr{SAK}_G.$

 For the reader's convenience we give the next.
\begin{defi} Let $R$ be a commutative ring with unit. An ideal $I$ of $R$ is called  \textit{regular} provided that it is generated by idempotents. 
We say that $I$ is \textit{max-regular} if it is maximal in the set of proper regular ideals of $R,$ ordered by inclusion.  We denote   by  $\mathfrak{mr}(R)$ the set of max-regular ideals of $R.$
\end{defi}

Let $X$ be a topological space and $R=C(X)$ the ring of continuous real valued functions of $X.$ It is clear that the idempotents in $R$ are characteristic functions $\chi_U$ on clopen subsets $U$ of $X.$ The collection of clopen subsets of $X$ will be denoted from now on by $Clop(X).$  Further,  let $H(X)$ be the group  of homeomorphisms of $X,$ 
$M$ an ideal  of $R$ and $\alpha\in H(X).$  We set $M\alpha=\{f\circ \alpha: f\in M\}.$ We proceed with the next.

\begin{lem}\label{regular} 
	Let  $X$ be a topological space, $\alpha\in H(X)$ and $M$ be an ideal of $R=C(X);$ then the following statements hold. \begin{enumerate}
		\item [(i)]  If $M$ is regular, then  $M\alpha$   is regular,
		\item [(ii)] If $M\in\mathfrak{mr}(R)$, then $M\alpha\in \mathfrak{mr}(R)$.
	\end{enumerate}
\end{lem}

\begin{proof}
	(i) It is clear that  $M\alpha$ is an additive subgroup of $C(X).$ 
 To prove that it is an ideal  take  $h\in R$ and  $f\in M.$ Denoting the product in $R$ by concatenation   we have for $x\in X$ that 
\begin{center}
		$[h(f\circ \alpha)](x)=h(x)(f\circ \alpha)(x)=h(\alpha^{-1}(\alpha(x)))f(\alpha(x))=[(h\circ \alpha^{-1})f]\circ \alpha(x),$
	\end{center}
then $h(f\circ\alpha)=[(h\circ\alpha^{-1})f]\circ\alpha\in M\alpha$, which shows that  $M\alpha$ is an ideal. Now we prove that  $M\alpha$ is regular. 
Since $M$ is regular there is a family of idempotents $\{\chi_U\}_{U\in\mathfrak{U} }$  that generates $M,$ for some $\mathfrak{U}\subseteq Clop(X).$ 
It is not difficult to see that $ \chi_U\circ \alpha=\chi_{ \alpha^{-1}(U)},$ for any $U\in\mathfrak{U}.$   
 We shall show that $\{\chi_{\alpha^{-1}(U)}\}_{U\in\mathfrak{U} }$ is a family of idempotents that generates  $M\alpha$. For this take  $f\in M$,  $U_1,\cdots,U_n\in\mathfrak{U} $ and  $f_1,\cdots,f_n\in R$ such that  $f=f_1\chi_{U_1}+\dots+f_n\chi_{U_n}$.  Then  for $x\in X$ we have 
\begin{align*}
(f\circ \alpha)(x)&=f_1(\alpha(x))\chi_{U_1}( \alpha(x))+\cdots+f_n(\alpha(x)) \chi_{U_n}(\alpha(x)),\\
&=f_1(\alpha(x))(\chi_{U_1}\circ \alpha)(x)+\cdots+f_n(\alpha(x))(\chi_{U_n}\circ \alpha)(x),\\
&=f_1(\alpha(x))\chi_{ \alpha^{-1}(U_1)}(x)+\cdots+f_n(\alpha(x)) \chi_{ \alpha^{-1}(U_n)}(x),
\end{align*}
	then $f\circ\alpha$ belongs to the ideal generated by $\{\chi_{\alpha^{-1}(U)}\}_{U\in\mathfrak{U} },$ thus $M\alpha$ is regular, as desired.
	
	(ii) By part (i)  the ideal $M\alpha$  is regular. We check that it is max-regular, for this  let $N$ be a regular ideal of $C(X)$ containing $M\alpha$ then $M=(M\alpha)\alpha^{-1}\subseteq N\alpha^{-1},$ since  $N\alpha^{-1}$ is regular we have     $M=N\alpha^{-1}$ then  $M\alpha=N$ which shows (ii). 
	
\end{proof}

Let $X$ be a topological space and $Comp(X)$ be  the collection of connected components of $X.$ According to  \cite{AT} the connected components of $X$ can be described using the ring $C(X).$  Let $M\in\mathfrak{mr}(C(X))$ then
\begin{equation}\label{cxm}C^X_M=\{x\in X: M\subseteq \mathfrak{m}_x\},
\end{equation} 
is a connected component of $X,$ where, for each $x\in X, \mathfrak{m}_x=\{f\in C(X)\mid f(x)=0\}$ is a maximal ideal of $C(X).$ We endow  $Comp(X)$ with the topology $\mathfrak{T}$ which has $\{\mathscr{O}_U\}_{U\in Clop(X)}$ as basis, being  $\mathscr{O}_U=\{C^X_M: \chi_U\notin M\},$ for any  $U\in Clop(X).$  Additionally, the topology $\mathfrak{T}$ makes the projection $\psi_X:X\rightarrow Comp(X)$ continuous.
 
 


 We recall the next.
\begin{pro}\label{T}{\rm  Let   $X$ be a   compact Hausdorff space. Then the following assertions hold.
\begin{itemize} 
\item [(i)] $(Comp(X),\mathfrak{T})$ is a profinite space,
\item [(ii)] $Comp(X) =\{C^X_M: M\in \mathfrak{mr}(C(X))\}.$ 

\end{itemize}
}
\end{pro}
\begin{proof} Item (i) is \cite[Theorem 3.8]{AT}  while (ii) is \cite[Remark 3.11]{AT}.
\end{proof}

When the space $X$ is clear  the set $C^X_M$ defined in \eqref{cxm} shall be denoted by $C_M.$ It is not difficult to see that $\alpha[C_M]=C_{M\alpha^{-1}},$ for any $\alpha \in H(X).$

\subsubsection{The construction of a functor $H: \mathscr{SAK}_G\to  \mathscr{SAP}_G$}

As mentioned  at the beginning of this section, we wish to define a functor $H: \mathscr{SAK}_G\to  \mathscr{SAP}_G$ which is left adjoint to the inclusion functor $\mathscr{SAP}_G\to \mathscr{SAK}_G. $ The next result tells us how $H$ acts at the level of objects.
\begin{lem}\label{lema 5.1} Let  $\alpha$ be a separately continuous action of a topological  group $G$  on a compact Hausdorff space $X$, $R=C(X)$ and $g\in G.$ Then the following assertions hold. \begin{enumerate}
\item [(i)]The map
\begin{equation}\label{hob}H(\alpha): G\times Comp(X)\ni (g,C_M)\mapsto C_{M\alpha_{g^{-1}}}\in  Comp(X),\end{equation} is a separately continuous  action of $G$ on $Comp(X),$
\item[(ii)]	If $G$ is Hausdorff Baire and the set $Clop(X)$ is countable, then $H(\alpha)$ is continuous. 
\end{enumerate}
\end{lem}
\begin{proof} (i) We first check that $H(\alpha)$ is an action.  Notice that for $(g,C_M)\in G\times Comp(X)$ then  $C_{M\alpha _{g^{-1}}}\in Comp(X).$	 Now to check that $H(\alpha)$ is well defined, take $g\in G, M,N\in \mathfrak{mr}(R)$ such that, $C_M=C_N$ and  $x\in C_{M\alpha _{g^{-1}}}$. Then $M\alpha _{g^{-1}}\subseteq \mathfrak{m}_x,$ but $\alpha_{g^{-1}}(x)\in C_M$ because for $f\in M$, $f(\alpha_{g^{-1}}(x))=(f\circ\alpha_{g^{-1}})(x)=0$, then $\alpha_{g^{-1}}(x)\in C_N,$ which implies 
$x\in C_{N\alpha _{g^{-1}}}$ and $C_{M\alpha _{g^{-1}}}\subseteq C_{N\alpha _{g^{-1}}}$. In a similar way one shows $C_{M\alpha _{g^{-1}}}\supseteq C_{N\alpha _{g^{-1}}}$ and $H(\alpha)$ is well defined.  It is clear that  $H(\alpha)(1_G, C_M)=C_M,$ for any $M\in \mathfrak{mr}(R).$  Take  $g,h\in G$ and  $M\in \mathfrak{mr}(R)$ then
$$H(\alpha)[g, H(\alpha)(h, C_M)]=H(\alpha)(g,C_{M\alpha _{h^{-1}}})=C_{(M\alpha _{h^{-1}})\alpha_{g^{-1}}}=C_{M\alpha _{(gh)^{-1}}}=H(\alpha)(gh,C_M).$$ Now we prove that $H(\alpha)$ is separately continuous. Fix $g\in G$ and consider the map $H_{\alpha_g}: Comp(X)\ni C_M\mapsto C_{M\alpha_{g^{-1}}}\in Comp(X).$ To show that it is continuous, let $U\in Clop(X),$ then  
\begin{align*}
		H_{\alpha_g}^{-1}[\mathscr{O}_U] &=\{C_M\in Comp(X): C_{M\alpha _{g^{-1}}}\in \mathscr{O}_U\}\\
&=\{C_M\in Comp(X): \chi_U\notin M\alpha _{g^{-1}}\}\\
&=\{C_M\in Comp(X): \chi_{\alpha_{g^{-1}}(U)}\notin M\}\\
&=\mathscr{O}_{\alpha_{g^{-1}}(U)},
\end{align*} and the result follows from the fact that $\{\mathscr{O}_U\}_{U\in Clop(X)}$ is a basis for $Comp(X)$. Now fix $C_M\in Comp(X)$ then we have a map $H^{C_M}:G\ni g\mapsto C_{M\alpha_{g^{-1}}}\in Comp(X).$ Take $x\in C_M,$ since $\alpha$ is a separately continuous action 
the map $\alpha^x:G\ni g\mapsto g\cdot x\in X$ is continuous, furthermore, if $\psi_X$ 
is the projection, then  $\psi_X(g\cdot x)=\alpha_g[C_M],$ because  $g\cdot x\in \psi_X(g\cdot x)\cap \alpha_g[C_M]=\psi_X(g\cdot x)\cap C_{M\alpha_{g^{-1}}},$  thus $\psi_X\circ \alpha^x= H^{C_M}$ and $H^{C_M}$ is continuous.

(ii) It follows from (i) of Proposition \ref{T} that  $(Comp(X),\mathfrak{T})$ is a profinite space, then $Comp(X)$ it is metrizable because the set  $\{\mathscr{O}_U\}_{U\in Clop(X)}$ is countable, then the result follows from (i) above   and  \cite[Theorem 3.1]{GPU}.
	\end{proof}

Now we show how $H$ acts at the level of morphisms. Let  $\alpha=\lbrace \alpha_g:X\rightarrow X\rbrace_{g\in G}$ and  $\beta=\lbrace \beta_g:Y\rightarrow Y\rbrace_{g\in G}$ be objects in $\mathscr{SAK}_G$ and $f:\alpha \to \beta$ a morphism.  Given 
 $C_M\in Comp(X)$ exists a unique $C\in Comp(Y)$ containing  $f(C_M),$ that is, there is   $N\in\mathfrak{mr}(C(Y))$ such that  $f(C_M)\subseteq C_N.$ With this in mind 
we define a map  
\begin{equation}\label{hmor}H(f):Comp(X)\ni C_M\mapsto C_N\in  Comp(Y).\end{equation} 
Observe that $H(f)$ is continuous because for   $U\in Clop(Y)$, one has  $H(f)^{-1}[\mathscr{O}_U]=\mathscr{O}_{f^{-1}(U)}.$ The following lemma guarantees that   $H(f)$ is a morphism in  $\mathscr{SAP}_G.$ 

\begin{lem}\label{lema 5.2} With the notations above  the map $H(f)$ is a morphism from $H(\alpha)$ to $H(\beta)$ in $\mathscr{SAP}_G$ and the diagram
	\[\xymatrix {
		X\ar[r]^{f}  \ar[d]_{\psi_X} & Y\ar[d]^{\psi_Y} \\
		Comp(X)\ar[r]_{H(f)} & Comp(Y) 
	}\]
is commutative, where the vertical arrows are the canonical projections.
\end{lem}
\begin{proof} First of all as observed in \cite[p 155]{AT} the commutativity of the diagram follows from  \cite[Remark 3.11]{AT}. Moreover, since $H(f)$ is continuous we only need to show that it is a  $G$-map. Take  $C_M\in Comp(X)$ and $g\in G,$ by the definition of $H_{\alpha_g}$ we have $H(f)(H_{\alpha_g}(C_M))=H(f)(C_{M\alpha_{g^{-1}}})$. Moreover,  let  $N\in \mathfrak{mr}(C(Y))$ such that  $f(C_M)\subseteq C_N$ then $H(f)(C_M)=C_N,$ and  $H_{\beta_g}(H(f)(C_M))=H_{\beta_g}(C_N)=C_{N\beta_{g^{-1}}}.$ Thus we need to show  $f(C_{M\alpha_{g^{-1}}})\subseteq C_{N\beta_{g^{-1}}}.$ 
 Let  $x\in C_{M\alpha_{g^{-1}}}$, then $M\alpha_{g^{-1}}\subseteq \mathfrak{m}_x,$ and $M\subseteq \mathfrak{m}_{\alpha_{g^{-1}(x)}}$ which gives  $\alpha_{g^{-1}}(x)\in C_M.$ Also, $f(\alpha_{g^{-1}}(x))\in f(C_M)\subseteq C_N,$ but $f(\alpha_{g^{-1}}(x))=\beta_{g^{-1}}(f(x)),$ therefore  $N\subseteq\mathfrak{m}_{\beta_{g^{-1}}(f(x))}.$ We have proved that $N\beta_{g^{-1}}\subseteq\mathfrak{m}_{f(x)},$ that is $f(x)\in C_{N\beta_{g^{-1}}}$ and  $f(C_{M\alpha_{g^{-1}}})\subseteq C_{N\beta_{g^{-1}}},$ from this we conclude that  $H(f)$ is a $G$-map, as desired.
\end{proof}

\begin{pro}\label{prop 5.2} The map 
	determined by equations \eqref{hob} and \eqref{hmor} is a covariant functor  $H:\mathscr{SAK}_G\rightarrow \mathscr{SAP}_G.$
\end{pro}
\begin{proof}
Let $\alpha$ and $\beta$ be objects in  $\mathscr{SAK}_G$ acting on the spaces $X$ and $Y$ respectively,  and $f\in Hom_{\mathscr{SAK}_G}(\alpha,\beta).$  
  We have that  $H(f)\in Hom_{\mathscr{SAP}_G}(H(\alpha),H(\beta))$ by Lemma \ref{lema 5.1} and Lemma \ref{lema 5.2}. Also,  if $id_X:X\rightarrow X$ is the identity morphism, then    $H(id_X)=id_{Comp(X)}.$ Thus we only need to check that for $\gamma^Z \in Ob(\mathscr{SAK}_G)$ and  $g\in Hom_{\mathscr{SAK}_G}(\beta,\gamma),$ the equality $H(g\circ f)=H(g)\circ H(f)$ holds. Indeed, take $C_M\in Comp(X)$ then there are  $N_1\in \mathfrak{mr}(C(Y))$ and  $N_2\in\mathfrak{mr}(C(Z))$ such that  $g(f(C_M))\subseteq g(C_{N_1})\subseteq C_{N_2},$ this gives $H(g\circ f)(C_M)=C_{N_2}.$ On the other hand exists $N_3\in\mathfrak{mr}(C(Z))$ for which  $g(C_{N_1})\subseteq C_{N_3}$ and  $H(g)(H(f)(C_M))=H(g)(C_{N_1})=C_{N_3},$ thus $C_{N_2}=C_{N_3}$ and we conclude that $H$ is  a covariant functor.
\end{proof}
Before proving the reflectivity  of  $\mathscr{SAP}_G$ in $\mathscr{SAK}_G$ we need to fix  some notations and facts.  Let $\mathscr{C}, \mathscr{D}$ be two categories, $F:\mathscr{C}\to \mathscr{D}$ and  $T:\mathscr{D}\to \mathscr{C}$ be two covariant functors   and {\bf Set}  be the category of sets. Denote by $Hom_{\mathscr{D}}(F(-),-):\mathscr{C}^{op}\times\mathscr{D}\rightarrow {\bf Set}$ the functor sending an object $(X,Y)$ into $Hom_{\mathscr{D}}(F(X),Y)$ and a morphism  $(f^{op}, g) :(X, Y)\to (X', Y')$ to the map 
$Hom_{\mathscr{D}}(F(f),g):Hom_{\mathscr{D}}(F(X),Y)\ni m\mapsto  g\circ m\circ F(f)\in  Hom_{\mathscr{D}}(F(X'),Y').$   In a similar way one defines  $Hom_{\mathscr{C}}(-,T(-)):\mathscr{C}^{op}\times\mathscr{D}\rightarrow {\bf Set},$ sending  the object $(X,Y)$  to $Hom_{\mathscr{C}}(X,T(Y))$ and a morphism  $(f^{op}, g) :(X, Y)\to (X', Y')$ to 
$Hom_{\mathscr{C}}(f^{op},T(g)):Hom_{\mathscr{C}}(X,T(Y))\ni m\mapsto  T(g)\circ m\circ f\in  Hom_{\mathscr{C}}(X',T(Y')).$

\begin{rem}
If $\alpha^X$ is an object in $\mathscr{SAK}_G$ it is not difficult to show that the projection  $\psi_X:X\rightarrow Comp(X)$ is a morphism between $\alpha^X$ and $H(\alpha)^{Comp(X)}$ in $\mathscr{SAK}_G.$ Further, $\psi_X$ is an  isomorphism provided that  $X$ is totally  disconnected. 
\end{rem}

By abuse of notation the set  $Hom_{\mathscr{SAK}_G}(\alpha^X,\beta^Y),$ shall be denoted by $Hom_{\mathscr{SAK}_G}(X,Y),$ for any $\alpha^X,\beta^Y$ objects in $\mathscr{SAK}_G.$

Now we prove one of the main results of this section.
\begin{teo} \label{reft}The category
 $\mathscr{SAP}_G$ is  reflective  in $\mathscr{SAK}_G$.
\end{teo}
\begin{proof}  We shall show that  $\tau:Hom_{\mathscr{SAP}_G}(H(-),-)\Rightarrow Hom_{\mathscr{SAK}_G}(-,\iota(-))$ sending  an object  $(\alpha^X,\beta^Y)\in (\mathscr{SAK}_G)^{op}\times(\mathscr{SAP}_G)$ to 
$$\tau_{(\alpha^X,\beta^Y)}:Hom_{\mathscr{SAP}_G}(Comp(X),Y) \ni g \mapsto g\circ \psi_X \in Hom_{\mathscr{SAK}_G}(X,Y),$$
being  $\iota: \mathscr{SAP}_G\to \mathscr{SAK}_G$ the inclusion functor; is a natural isomorphism. The injectivity of  $\tau_{(\alpha^X,\beta^Y)}$ follows from the fact that   $\psi_X$ is surjective. Take  $g\in Hom_{\mathscr{SAK}_G}(X,Y)$ and note that $\psi_Y$ is a homeomorphism since $Y$ is profinite. 
Setting $h:=\psi_Y^{-1}\circ H(g)\in Hom_{\mathscr{SAP}_G}(Comp(X),Y)$ one has from  Lemma \ref{lema 5.2} that  $\tau_{(\alpha^X,\beta^Y)}(h)=g,$ and $\tau_{(\alpha^X,\beta^Y)}$ is surjective. 
It remains to show that the following diagram is commutative
	\[\xymatrix {
	Hom_{\mathscr{SAP}_G}(Comp(X),Y) \ar[r]_{f_2\circ(-)\circ H(f_1)}  \ar[d]_{\tau_{(\alpha^X,\beta^Y)}} & Hom_{\mathscr{SAP}_G}(Comp(X'),Y')\ar[d]^{\tau_{(\mu^{X'},\gamma^{Y'})}} \\
		Hom_{\mathscr{SAK}_G}(X,Y) \ar[r]_{f_2\circ(-)\circ f_1} & Hom_{\mathscr{SAK}_G}(X',Y'), 
	}\]
 where $(\alpha^X,\beta^Y), (\mu^{X'},\gamma^{Y'})$  are objects in $(\mathscr{SAK}_G)^{op}\times(\mathscr{SAP}_G), f_1^{op}:\alpha^X\rightarrow \mu^{X'}$ and $f_2: \beta^Y\rightarrow \gamma^{Y'}$ are morphisms in $(\mathscr{SAK}_G)^{op}$ and $\mathscr{SAP}_G$ respectively. Let  $g:Comp(X)\rightarrow Y$ a morphism in $\mathscr{SAP}_G$ then
	$$\tau(\mu^{X'},\gamma^{Y'})\circ Hom_{\mathscr{SAP}_G}(H(f_1),f_2)(g)=f_2\circ g\circ H(f_1)\circ\psi_{X'},$$
	but by  Lemma \ref{lema 5.2} we have  $\psi_X\circ f_1=H(f_1)\circ \psi_{X'},$ hence 
	$$f_2\circ g\circ H(f_1)\circ \psi_{X'}=f_2\circ g\circ \psi_X\circ f_1=Hom_{\mathscr{SAK}_G}(f_1^{op},f_2)\circ \tau(\alpha^X,\beta^Y)(g).$$ We have proved that  $\tau$ is a natural isomorphism, therefore   $\mathscr{SAP}_G$ is reflective in $\mathscr{SAK}_G$ which finishes the proof.\end{proof}

Let $\mathscr{AKC}_G$ be the  subcategory of  $\mathscr{SAK}_G$  having as objects continuous  actions on compact Hausdorff spaces having countable clopen sets, analogously one defines the subcategory  $\mathscr{APC}_G$ of  $\mathscr{SAP}_G$.  We finish our work with the next.
\begin{pro}\label{ref2} Let $G$ be a Hausdorff and Baire group. Then the category
 $\mathscr{APC}_G$ is  reflective in $\mathscr{AKC}_G$.
\end{pro}
\begin{proof} Let $X$ be a compact Hausdorff space having countable clopen sets, since $\psi_X:X\rightarrow Comp(X)$ is a continuous surjection, then  $Comp(X)$ has countable clopen sets and  equations \eqref{hob} and \eqref{hmor} and item (ii) of Lemma \ref{lema 5.1}  determine a functor $H:\mathscr{AKC}_G\rightarrow \mathscr{APC}_G,$ also  the proof that 
\begin{equation}\label{taup}\mathcal{T}:Hom_{\mathscr{APC}_G}(H(-),-)\Rightarrow Hom_{\mathscr{AKC}_G}(-,\iota(-)),\end{equation} sending  an object  $(\alpha^X,\beta^Y)\in (\mathscr{AKC}_G)^{op}\times(\mathscr{APC}_G)$ to 
$$\mathcal{T}_{(\alpha^X,\beta^Y)}:Hom_{\mathscr{APC}_G}(Comp(X),Y) \ni g \mapsto g\circ \psi_X \in Hom_{\mathscr{AKC}_G}(X,Y),$$
being  $\iota: \mathscr{APC}_G\to \mathscr{AKC}_G$ the inclusion functor; is a natural isomorphism, is the same as in Theorem \ref{reft}. Thus,  $H$ is left adjoint to the inclusion functor $\iota: \mathscr{APC}_G\to \mathscr{AKC}_G.$ 
\end{proof}

	\end{document}